\documentclass[12pt]{article}
\usepackage[utf8]{inputenc} 
\usepackage[T1]{fontenc}    
\usepackage{hyperref}       
\usepackage{url}            
\usepackage{booktabs}       
\usepackage{amsfonts}       
\usepackage{nicefrac}       
\usepackage{microtype}      
\usepackage{lipsum}
\usepackage{arxiv}
\usepackage{amsmath,amssymb}
\usepackage{mathtools}
\usepackage{enumitem}
\usepackage{amsthm}
\usepackage{xcolor}
\usepackage{pgfplots}
\pgfplotsset{compat=1.14}
\usepackage{subfig}
\usepackage[noabbrev]{cleveref}
\newtheorem{theorem}{Theorem}[section]
\crefname{theorem}{Theorem}{Theorems}
\newtheorem{lemma}[theorem]{Lemma}
\crefname{lemma}{Lemma}{Lemmas}
\newtheorem{proposition}[theorem]{Proposition}
\crefname{proposition}{Proposition}{Propositions}
\newtheorem{corollary}[theorem]{Corollary}
\crefname{corollary}{Corollary}{Corollaries}
\theoremstyle{definition}

\crefname{defn}{Definition}{Definitions}
\newtheorem{example}[theorem]{Example}
\crefname{example}{example}{examples}
\newtheorem{rmk}[theorem]{Remark}
\crefname{rmk}{remark}{remarks}
\Crefname{rmk}{Remark}{Remarks}
\title{Growth on multiple interactive-essential resources in a self-cycling fermentor: An impulsive differential equations approach}
\author{Tyler Meadows \\
    Department of mathematics and statistics \\
    McMaster University\\
    Hamilton, Ontario, L8S 4L8\\
    \texttt{meadowta@mcmaster.ca}
    \And Gail S.K. Wolkowicz\\
    Department of mathematics and statistics\\
    McMaster University\\
    Hamilton, Ontario, L8S 4L8\\
    \texttt{wolkowic@mcmaster.ca}}
\newcommand{\IN}{{\rm{in}}}
\renewcommand{\vec}[1]{\mathbf{#1}}
\newcommand{\change}[1]{\color{black} #1 \color{black}}
\begin{document}
\maketitle
\begin{abstract}
        We introduce a model of the growth of a single microorganism in a self-cycling fermentor in which an arbitrary number of resources are limiting, and impulses are triggered when the concentration of one specific substrate reaches a predetermined level. The model is in the form of a system of impulsive differential equations. We consider the operation of the reactor to be successful if it cycles indefinitely without human intervention and derive conditions for this to occur. In this case, the system of impulsive differential equations has   a  periodic solution. We show that success is equivalent to the convergence of solutions to this periodic solution. We provide conditions that ensure that a periodic solution exists. When it exists, it is unique and attracting.  However, we also show that whether a solution converges to this periodic solution, and hence whether the model predicts that the reactor operates successfully, is  initial condition dependent. The analysis is illustrated with numerical examples.
\end{abstract}

\section{Introduction}

The self-cycling fermentation (SCF) process can be described as a sequential batch process and is an example of a hybrid system. In SCF, a tank is filled with a liquid medium that contains nutrients and microorganisms that use these nutrients to grow. The liquid medium is mixed to keep the concentrations uniform while the microorganisms feed on the nutrients and grow. If a predetermined decanting criterion is met, the tank is partially drained and subsequently refilled with fresh medium. Many different decanting criteria can be used to initiate the emptying/refilling sequence, such as elapsed time, a specific nutrient concentration, or a specific biomass concentration. For example, in \cite{Wincure1995}, a specific dissolved oxygen concentration was used as the decanting criterion. The goal was to choose the decanting criterion so that the fermentor would run indefinitely without operator input.

\change{Self-cycling fermentors and sequential batch reactors are often used to improve the efficiency of wastewater-treatment facilities \cite{Hughes1996,Sarkis1994}, to cultivate microorganisms \cite{Sauvageau2010}, to produce some biologically derived compounds \cite{Storms2012,Wang2017}, and as a method of producing bacteriophages for use in phage therapy \cite{Sauvageau2010a}. In particular, the self-cycling fermentation process has been suggested as an addition to the sidestream partial nitritation process in order to reduce the competition pressure on the beneficial Anammox bacteria \cite{Strous1998, laureni2016}.

Traditionally, the nitrification process is done in multiple stages; Ammonium ($\rm{NH_4^+}$) is converted  to nitrite ($\rm{NO_2^-}$) by ammonium oxidizing bacteria (AOB), nitrite is converted to nitrate ($\rm{NO_3^-}$) by nitrite oxidizing bacteria (NOB), and nitrate is converted to dinitrogen gas ($\rm{N_2}$) by denitrifying bacteria. Anammox bacteria offer a shortcut in which ammonium and nitrite are converted directly to dinitrogen gas. Each stage occurs in a continuous flow reactor. Unfortunately, Anammox is limited by both ammonium and nitrite and its growth is slow, allowing NOB to easily outcompete Anammox for nitrite. Self-cycling fermentation (in combination with biofilm cultivation) has been suggested as one way to tilt the competition in Anammox's favor \cite{laureni2016}.}

The decanting criterion can have a profound effect on the successful operation of the reactor. If the decanting criterion is too strict (e.g., complete removal of a resource), it may never be reached, and if it is too lenient (e.g., a small increase in biomass concentration), it may be reached too often. Many studies have modelled the growth of a single species with a single limiting resource with different decanting criteria, such as: threshold biomass concentrations \cite{Sun2011}; threshold nutrient concentrations \cite{Fan2007,Smith2001}; or after a certain time elapsed that depends on the nutrient concentrations after the previous decanting stage \cite{Cordova2014}. Under the assumption that the emptying/refilling process occurs on a much faster time scale than the other processes in the system, the system can be modelled using a system of impulsive differential equations. For a discussion on the qualitative theory of impulsive differential equations see \cite{Bauinov1995,Samoilenko1995}.

A more recent paper by Hsu et al. \cite{Hsu2019} investigated the dynamics of a model with two essential limiting nutrients in which the decanting criterion required both nutrient concentrations to reach or be below a prescribed threshold. When modelling with multiple resources, two resources are said to be \textit{essential} if the microorganism cannot grow without both resources. Conversely, two resources are said to be \textit{substitutable} if the presence of either resource is enough to promote growth. The different ways in which a species may respond to multiple limiting nutrients exist on a spectrum that was described in the book by Tilman \cite{Tilman1982}. \change{In particular, essential nutrients may be further refined into \emph{perfectly-essential} nutrients and \emph{interactive-essential} nutrients based on their respective growth isoclines. The growth iscolines of two perfectly-essential nutrients meet at a right angle, indicating that one resource may not be substituted for the other. The growth isoclines of interactive-essential nutrients have a curved corner, indicating that there is a small range of nutrient concentrations for which partial substitution is possible. }

In \cite{Hsu2019}, nutrient uptake of two essential resources was modelled using Liebig's law of the minimum \cite{Liebig1840}, where the growth is limited by the nutrient concentration that results in the slowest individual growth rate. Many more modern engineering papers do not use Liebig's law and instead model nutrient uptake for essential nutrients using the product of individual uptake functions \cite{Bader1978}. This may be problematic in the case when a large number of resources are growth limiting; the product of many uptake functions may predict much lower growth than what is actually observed if each uptake function is a small number. However, the product of uptake functions is advantageous because it is differentiable, whereas the minimum of uptake functions given by Liebig's law of the minimum is only Lipschitz continuous. 

Implementation of a self-cycling fermentor can be difficult. Online measurements can be expensive, and measuring quantities of interest may be impractical. Operators of these reactors will often choose to make easier measurements that act as a proxy for the quantities of true interest. For example, in \cite{Wincure1995}, the authors measured the dissolved oxygen concentration, since it was known to reach a minimum at the same time as the limiting substrate was exhausted. \change{ In \cite{laureni2016}, the ammonium concentration was used as a threshold, even though both ammonium and nitrite were growth limiting.} Alternatively, operators may not be aware that some nutrient concentrations are lower than required in the input medium, and, as a result, unanticipated resources may become limiting.

In this paper, we investigate the growth of a single microorganism with an arbitrary number of essential nutrients in a self-cycling fermentor. The decanting criterion is met when one specific tracked nutrient concentration falls below a prescribed threshold value. We model nutrient uptake using a general class of functions that includes both the product of uptake functions used in much of the engineering literature and the minimum of uptake functions preferred by biologists. In the case with a single limiting resource, this model reduces to that given in \cite{Smith2001}. In the case with two essential limiting resources and nutrient uptake modelled using Liebig's law of the minimum, this model is the same as the one in \cite{Hsu2019} where one threshold concentration is arbitrarily large. 

The paper is organized as follows. In \cref{Sec:Model}, we introduce the model and show that it is mathematically and biologically well-posed. In \cref{Sec:Periodic}, we provide conditions for the system to have a unique periodic solution and find the basin of attraction for the periodic solution. We show that if the initial conditions lie outside of the basin of attraction, then the population of microorganisms will eventually die out, and the reactor will fail. In \cref{Sec:Conclusions}, we summarize what we have learned, compare with similar models and discuss what implications this may have for operators of self-cycling reactors. \change{Our analysis is supplemented by several examples with parameters chosen to illustrate specific results.} 
\section{The Model}\label{Sec:Model} 

We model the self cycling fermentor using the system of impulsive differential equations
\begin{subequations}\label{eq:model}
\begin{align}
\begin{drcases}
    \dot{s}_i(t) &= -\frac{1}{y_i}F(\vec{s}(t))x(t),\quad i = 1,\dots,n\\
    \dot x(t) & = (-D +F(\vec{s}(t)))x(t)
\end{drcases} && \vec{s}(t_k^-)\notin \Gamma^-,\\
\begin{drcases}
    \vec{s}(t_k^+) &= r\vec{s}^\IN+(1-r)\vec{s}(t_k^-) \hphantom{i=1,\dots,n}\\
    x(t_k^+) &= (1-r)x(t_k^-)
    \end{drcases} && \vec{s}(t_k^-)\in \Gamma^-,\label{eq:Impulses}
\end{align}
\end{subequations}
where $\vec{s}(t)=(s_1(t),\dots,s_n(t))^T$.  Here, $s_i(t)$ denotes the concentration of the $i$th nutrient and $x(t)$  denotes the concentration of the biomass in the tank at time $t$.

The set $\Gamma^-$ is called the \emph{impulsive set}, and it represents the condition on $\vec{s}$ that triggers the emptying/refilling process. We consider the case where only one of the  nutrients is tracked by the operator and the tank is reset when the concentration of this nutrient reaches a prescribed threshold. Without loss of generality, we  label this nutrient  $s_1$ and denote the prescribed threshold by $\overline{s_1}$. Therefore, we define the impulsive set 
\begin{equation}\Gamma^- = \{\vec{s}\in \mathbb{R}^n_+ : s_1 = \overline{s_1}\}.\end{equation}
This is an $(n-1)$-dimensional hyperplane restricted to the positive cone, $\mathbb{R}^n_+=\{z\in\mathbb{R}^n: z_i >0~\text{for}~i=1,...,n\}$. For simplicity, we assume that $s_1(0)> \overline{s_1}$. The \emph{impulse times} are then the times $\{t_k\}$ such that $\vec{s}(t_k^-)\in \Gamma^-$, where $\vec{s}(t_k^-) = \lim_{t\to t_k^-} \vec{s}(t)$. 

The parameter $D$ is the decay rate (or maintenance coefficient) for the microorganism $x$, $\vec{s}^\IN = (s_1^\IN,\dots,s_n^\IN)^T$, where $s_i^\IN$ is the concentration of the $i$th nutrient in the fresh medium, $r\in(0,1)$ is the fraction of the tank that is decanted and subsequently refilled, and $y_i>0$,  $i=1,\dots,n$, are the yield coefficients for each nutrient.  

We assume $F:\mathbb{R}^n_+\to\mathbb{R}_+$ is a Lipschitz-continuous function satisfying $F(\vec{s})=0$ if $s_i = 0$ for any $i=1,...,n$, $F(\vec{s})>0$ if every $s_i>0$, and increasing in each of its arguments (i.e., $F(\vec{s}+\varepsilon \vec{e}_i)>F(\vec{s})$ for any $\varepsilon>0$, where $\vec{e}_i$ is the $i$th positive unit vector in $\mathbb{R}^n$).

This class of functions includes Liebig's minimum function,
\begin{equation} \label{eq:Liebig}
    F(\vec{s}) = \min\{f_i(s_i):~i=1,...,n\},
\end{equation}
as well as the product of functions
\begin{equation} \label{eq:product}
    F(\vec{s}) = \prod_{i=1}^n f_i(s_i), 
\end{equation}
where each $f_i(s_i)$ denotes the rate at which the microorganism uptakes the $i$th nutrient and are assumed to be increasing, \change{Lipshitz continuous functions.} In 
Tilman's classification of resource types \cite{Tilman1982}, Liebig's minimum function \eqref{eq:Liebig} describes perfectly-essential nutrients \change{(level sets are shown in figure \ref{fig:PerfectlyEssential})}, and the product of functions \eqref{eq:product} describes interactive-essential nutrients \change{(level sets are shown in figure \ref{fig:InteractiveEssential})}.
In the engineering literature, it is common to use the Monod growth function, $f_i(s_i) = \frac{\mu_is_i}{k_i+s_i}$ to describe the uptake of the $i$th nutrient.
\begin{figure}[t]
    \centering
    \subfloat[]{\begin{tikzpicture}[]
\begin{axis}[height = {38.1mm}, ylabel = {$s_2$}, xmin = {0}, xmax = {5}, ymax = {5}, xlabel = {$s_1$}, domain = 0:10, restrict y to domain = 0:20 = {unbounded coords=jump,scaled x ticks = false,xlabel style = {font = {\fontsize{11 pt}{14.3 pt}\selectfont}, color = {rgb,1:red,0.00000000;green,0.00000000;blue,0.00000000}, draw opacity = 1.0, rotate = 0.0},xmajorgrids = true,xtick = {0.0,1.0,2.0,3.0,4.0,5.0},xticklabels = {$0$,$1$,$2$,$3$,$4$,$5$},xtick align = inside,xticklabel style = {font = {\fontsize{8 pt}{10.4 pt}\selectfont}, color = {rgb,1:red,0.00000000;green,0.00000000;blue,0.00000000}, draw opacity = 1.0, rotate = 0.0},x grid style = {color = {rgb,1:red,0.00000000;green,0.00000000;blue,0.00000000},
draw opacity = 0.1,
line width = 0.5,
solid},axis x line* = left,x axis line style = {color = {rgb,1:red,0.00000000;green,0.00000000;blue,0.00000000},
draw opacity = 1.0,
line width = 1,
solid},scaled y ticks = false,ylabel style = {font = {\fontsize{11 pt}{14.3 pt}\selectfont}, color = {rgb,1:red,0.00000000;green,0.00000000;blue,0.00000000}, draw opacity = 1.0, rotate = 0.0},ymajorgrids = true,ytick = {0.0,1.0,2.0,3.0,4.0,5.0},yticklabels = {$0$,$1$,$2$,$3$,$4$,$5$},ytick align = inside,yticklabel style = {font = {\fontsize{8 pt}{10.4 pt}\selectfont}, color = {rgb,1:red,0.00000000;green,0.00000000;blue,0.00000000}, draw opacity = 1.0, rotate = 0.0},y grid style = {color = {rgb,1:red,0.00000000;green,0.00000000;blue,0.00000000},
draw opacity = 0.1,
line width = 0.5,
solid},axis y line* = left,y axis line style = {color = {rgb,1:red,0.00000000;green,0.00000000;blue,0.00000000},
draw opacity = 1.0,
line width = 1,
solid},    xshift = 0.0mm,
    yshift = 0.0mm,
    axis background/.style={fill={rgb,1:red,1.00000000;green,1.00000000;blue,1.00000000}}
,legend style = {color = {rgb,1:red,0.00000000;green,0.00000000;blue,0.00000000},
draw opacity = 1.0,
line width = 1,
solid,fill = {rgb,1:red,1.00000000;green,1.00000000;blue,1.00000000},fill opacity = 1.0,text opacity = 1.0,font = {\fontsize{8 pt}{10.4 pt}\selectfont}},colorbar style={title=}}, ymin = {0}, width = {76.2mm}]\addplot+ [color = {rgb,1:red,0.50196078;green,0.50196078;blue,0.50196078},
draw opacity = 1.0,
line width = 1.5,
solid,mark = none,
mark size = 2.0,
mark options = {
            color = {rgb,1:red,0.00000000;green,0.00000000;blue,0.00000000}, draw opacity = 1.0,
            fill = {rgb,1:red,0.50196078;green,0.50196078;blue,0.50196078}, fill opacity = 1.0,
            line width = 1,
            rotate = 0,
            solid
        },forget plot]coordinates {
(0.4, 5.0)
(0.4, 0.4)
(5.0, 0.4)
};
\addplot+ [color = {rgb,1:red,0.50196078;green,0.50196078;blue,0.50196078},
draw opacity = 1.0,
line width = 1.5,
solid,mark = none,
mark size = 2.0,
mark options = {
            color = {rgb,1:red,0.00000000;green,0.00000000;blue,0.00000000}, draw opacity = 1.0,
            fill = {rgb,1:red,0.50196078;green,0.50196078;blue,0.50196078}, fill opacity = 1.0,
            line width = 1,
            rotate = 0,
            solid
        },forget plot]coordinates {
(1.0, 5.0)
(1.0, 1.0)
(5.0, 1.0)
};
\addplot+ [color = {rgb,1:red,0.50196078;green,0.50196078;blue,0.50196078},
draw opacity = 1.0,
line width = 1.5,
solid,mark = none,
mark size = 2.0,
mark options = {
            color = {rgb,1:red,0.00000000;green,0.00000000;blue,0.00000000}, draw opacity = 1.0,
            fill = {rgb,1:red,0.50196078;green,0.50196078;blue,0.50196078}, fill opacity = 1.0,
            line width = 1,
            rotate = 0,
            solid
        },forget plot]coordinates {
(2.0, 5.0)
(2.0, 2.0)
(5.0, 2.0)
};
\addplot+ [color = {rgb,1:red,0.50196078;green,0.50196078;blue,0.50196078},
draw opacity = 1.0,
line width = 1.5,
solid,mark = none,
mark size = 2.0,
mark options = {
            color = {rgb,1:red,0.00000000;green,0.00000000;blue,0.00000000}, draw opacity = 1.0,
            fill = {rgb,1:red,0.50196078;green,0.50196078;blue,0.50196078}, fill opacity = 1.0,
            line width = 1,
            rotate = 0,
            solid
        },forget plot]coordinates {
(3.0, 5.0)
(3.0, 3.0)
(5.0, 3.0)
};
\end{axis}

\end{tikzpicture}\label{fig:PerfectlyEssential}}
    \subfloat[]{\input{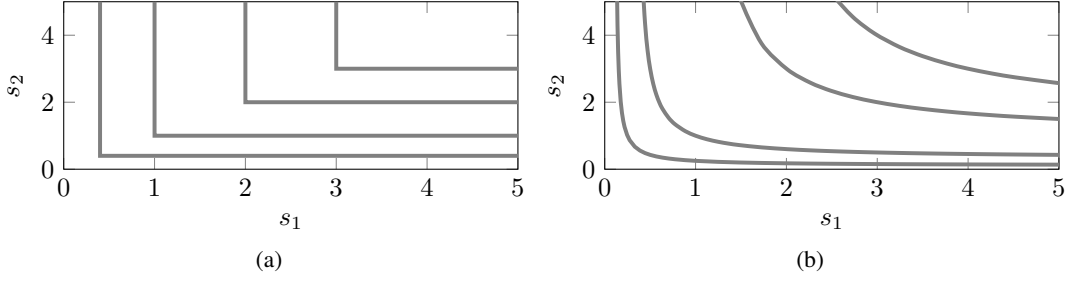}\label{fig:InteractiveEssential}}
    \caption{(a) Level sets of two perfectly-essential nutrients, as described by \cref{eq:Liebig}. (b) Level sets of two interactive-essential nutrients, as described by \cref{eq:product}. In both cases, $f_1= \frac{s_1}{1+s_1}$ and $f_2 = \frac{s_2}{1+s_2}$.}
    \label{fig:my_label}
\end{figure}

For $\vec{s}\notin \Gamma^-$ the system is governed by the system of ordinary differential equations,
\begin{subequations}\label{eq:ODES}
\begin{align}
    \dot{s}_i(t) &= -\frac{1}{y_i}F(\vec{s}(t))x(t),\quad i = 1,\dots,n, \label{eq:ODESa} \\
    \dot x(t) & = (-D +F(\vec{s}(t)))x(t). \label{eq:ODESb}
\end{align}
\end{subequations}

\begin{lemma}\label{lemma:odes} Solutions of \cref{eq:ODES} with initial conditions $(s_1(0),\dots,s_n(0),x(0)) \in \mathbb{R}_+^{n+1}$ are  bounded and satisfy $\vec{s}(t)\in\mathbb{R}_+^n$ for all $t\geq 0$. Furthermore, $x(t)\to0$ as $t\to\infty$.
\end{lemma}
\begin{proof}
Noting that $F(\vec{s})=0$ if $s_i=0$ for any $i=1,\dots,n,$ the faces of $\mathbb{R}_+^{n+1}$ are invariant, i.e., if $s_i(t)=0$, then $\dot{s}_i(t)=0$ and if $x(t)$, then $\dot{x}(t)=0$ . Since the vector field in \eqref{eq:ODES} is Lipschitz, solutions to initial value problems are unique \change{by the Picard-Lindel\"of theorem.} Therefore, any solution with initial conditions in the interior of $\mathbb{R}_+^{n+1}$ is confined to the interior of $\mathbb{R}_+^{n+1}$, otherwise it would intersect the faces of $\mathbb{R}_+^{n+1}$. The right hand side of each nutrient equation is non-positive, and so the nutrient concentrations are nonincreasing, which implies that $F(\vec{s}(t))$ is a nonincreasing function of $t$.

If $x(0)>0$, then there exists $t_*\geq0$ such that $F(\vec{s}(t))<D$ for all $t\geq t_*$. If not, then $F(\vec{s}(t))\geq D$ for all $t$, and therefore 
\begin{align*}
    x'(t) =(F(\vec{s}(t))-D)x(t)\geq 0.
\end{align*}
Since $x(t)$ is nondecreasing, it follows that $x(t)\geq x(0)$ for all $t$. Therefore,
\begin{align*}
    s_i'(t) &\leq -\frac{1}{y_i} D x(0).
\end{align*}
This implies that $s_i(t) \leq s_i(0)-\frac{1}{y_i}Dx(0)t$ for all $t\geq 0$, and hence $s_i(t)\to -\infty$ as $t\to \infty$, a contradiction.

Therefore, there exists $t_*\geq0$ such that $F(\vec{s}(t_*))<D$ for all $t\geq t_*$.  This implies  that 
\begin{align*}
    x'(t) \leq (F(\vec{s}(t_*))-D)x(t)<0,
\end{align*}
for all $t\geq t_*$. Integrating gives 
\begin{align*}
x(t) \leq x(t_*) e^{(F(\vec{s}(t_*))-D)(t-t_*)}.
\end{align*}
Therefore, $x(t)\to 0$ as $t\to \infty$.
\end{proof}

Dividing the other nutrient equations in \eqref{eq:ODESa} by the equation for $s_1(t)$
(i.e., considering $\dot{s_i}/\dot{s}_1,~i=2,\dots,n$)  and integrating, it follows that the nutrient concentrations are linear functions of $s_1(t)$. In vector form, 
\begin{align}
    \vec{s}(t) &= \vec{s}^0 - y_1(s_1^0-s_1(t))\vec{Y},
\end{align}
where $\vec{Y} =(1/y_1,\dots,1/y_n)^T$ and $\vec{s}^0 = (s_1(0),\dots,s_n(0))^T$. Note that the equation for $s_1$ in this form is trivial. For positive initial conditions, $s_1(t)$ is strictly decreasing as a function of time, and so $s_1(t)$ is invertible, allowing us to write $t(s_1)$. \change{This means that there is a one-to-one correspondance between the time $t$ and $s_1$. In a sense this allows us to use the nutrient concentration $s_1$ to measure time.} With this in mind, we can write
\begin{subequations}\label{eq:Sols}
    \begin{align}
    \vec{s}(s_1)&= \vec{s}^0-y_1(s_1^0-s_1)\vec{Y},\\
    x(s_1) &=x^0-y_1\int_{s_1^0}^{s_1}\left( 1-\frac{D}{F(\vec{s}(\tau))}\right)d\tau,\label{eq:xnoimpulses}
\end{align}
\end{subequations}
where \change{$x^0$ is the initial biomass concentration} \change{and \eqref{eq:xnoimpulses} follows by dividing \eqref{eq:ODESa} by the $s_1$ version of \eqref{eq:ODESb} and integrating with respect to $s_1$.} \change{Note that the notation is consistent since $x(s_1^0)=x^0$.} If there exists $t_1$ such that $s_1(t_1^-) = \overline{s_1}$, then we can reparameterize \eqref{eq:Sols} using the percentage of $s_1$ consumed up to that point. Let $\nu(s_1) = (s_1^0-s_1)/(s^0_1-\overline{s_1}).$  \change{Substituting $\nu \in [0,1]$ into \eqref{eq:Sols} gives}
\begin{align*}
    \vec{s}(\nu) &= \vec{s}^0 - \nu y_1(s_1^0-\overline{s_1})\vec{Y},\\
    x(\nu) &=x^0+ y_1(s_1^0-\overline{s_1})\int_0^\nu\left( 1-\frac{D}{F(\vec{s}(\tau))}\right)d\tau.
\end{align*}
After the first impulse,  $s_1 \in [\overline{s_1},\overline{s_1}^+]$, where $\overline{s_1}^+ = rs_1^\IN +(1-r)\overline{s_1}$ is the image of $\overline{s_1}$ under the impulsive map. In general, for each $k\geq 1$ for which there exists $t_k^-$ such that $s_1(t_k^-) = \overline{s_1}$, we write
\begin{align}
    \varphi_\nu(\vec{s}^k) &= \vec{s}^k -\nu y_1(s_1^k-\overline{s_1})\vec{Y},\label{eq:S}\\
    u_\nu(\vec{s}^k,x^k) &= x^k +y_1(s_1^k-\overline{s_1})\int_0^\nu\left( 1-\frac{D}{F(\varphi_\tau(\vec{s}^k))}\right)d\tau,\label{eq:x}
\end{align}
with the understanding that $s_1^k = \overline{s_1}^+$.
In this notation,
$$ \varphi_0(\vec{s}^k) =\vec{s}^k = \vec{s}(t_k^+) \quad \mbox{and} \quad \varphi_1(\vec{s}^k) = \vec{s}(t_{k+1}^-).$$ 
$$ u_0(\vec{s}^k,x^k) =x^k =  x(t_k^+) \quad \mbox{and} \quad u_1(\vec{s}^k,x^k) = x(t_{k+1}^-).$$

First we prove that if there are an infinite number of impulses, then the reactor cycles indefinitely with finite cycle time. I.e., the phenomenon of beating is not possible for system \eqref{eq:model}. 
\begin{lemma}\label{lemma:beating} 
Assume that $(s_1(t),\dots,s_n(t),x(t))\in \mathbb{R}_+^{n+1}$ is a solution to \eqref{eq:model} with an infinite number of impulse times $\{t_k\}_{k=1}^\infty$. Then $\lim_{k\to\infty} t_k =\infty$. 
\end{lemma}
\begin{proof} Since the $s_i$ are strictly decreasing, if $x(t)>0$, we can solve the $s_1$ equation in \cref{eq:model} for the time between impulses (i.e., consider $dt/ds_1$ and again use the substitution $\nu(s_1) = (s_1^0-s_1)/(s^0_1-\overline{s_1})$). After the first impulse, the time between impulses is given by
\begin{equation*}
    t_{k+1}-t_k = y_1(\overline{s_1}^+-\overline{s_1})\int_0^1\frac{1}{F(\varphi_\nu(\vec{s}^k))u_\nu(\vec{s}^k,x^k)}d\nu.
\end{equation*}
In order to show that the sequence $\{t_k\}_{k=1}^\infty$ has no accumulation point, it is enough to show that there exists $M>0$, independent of $k$, such that $F(\varphi_\nu(\vec{s}^k))u_\nu(\vec{s}^k,x^k)<M$. 
 For $\nu\in[0,1]$, each component of $\varphi_\nu(\vec{s}^k)$ is decreasing in $\nu$; i.e.,  
$$(\varphi_\nu)_i(\vec{s}^k) \leq(\varphi_0)_i(\vec{s}^k) =s_i^k$$
for $\nu\in[0,1]$, where $(\varphi_\nu)_i$ is the $i$th component of $\varphi_\nu$, $i>1$. By the relationship, $s_i^k = rs_i^\IN+(1-r)(\varphi_1)_i(\vec{s}^{k-1})$, for $i>1$, we obtain 
$$s_i^{k+1} \leq rs_i^\IN+(1-r)s_i^k.$$
Let $\{q_i^k\}_{k=0}^\infty$ be the sequence defined by $q_i^0 = s_i^0$, $q_i^{k+1} = rs_i^\IN+(1-r)q_i^k$. Then,  
\begin{equation}
    \limsup_{t\to\infty} s_i(t) \leq \lim_{k\to\infty}\sup_{\nu\in[0,1]}(\varphi_\nu)_i(\vec{s}^k)\leq \lim_{k\to \infty}q_i^k= s_i^\IN, 
\end{equation}
\noindent
and thus each $s_i(t)$ is bounded above. It remains to show that $x(t)$ is bounded.  
By \cref{eq:x}, there exists $M_0>0$ such that
\begin{equation*}
u_\nu(\vec{s}^k,x^k) \leq x^k+M_0, \quad \text{for all}~ \nu \in [0,1]. 
\end{equation*}
Using the relations $u_1(\vec{s}^k,x^k) =x(t_{k+1}^-)$ and $x^k = (1-r)x(t_k^-)$, it follows that
\begin{equation}
    \frac{1}{1-r}x^{k+1}=x{(t_{k+1}^-)} = u_1(s^k,x^k)\leq x^k+M_0
\end{equation}
and hence
\begin{equation}
    x^{k+1} \leq (1-r)(x^k+M_0).
\end{equation}
Consider the sequence $\{y_k\}_{k=0}^\infty$, defined by $y(0) = x^0$ and  $y_{k+1} = (1-r)(y_k+M_0)$, for $k=1,2,\dots.$ Then
\begin{equation*}
    \limsup_{t\to\infty}x(t) \leq \lim_{k\to\infty}\sup_{\nu \in [0,1]}u_\nu(\vec{s}^k,x^k)\leq \lim_{k\to\infty} y_k = \frac{(1-r)M_0}{r}. \qedhere
\end{equation*}
\end{proof}
\begin{corollary}\label{lemma:Bio_well_posed} 
Let $(s_1(t),\dots,s_n(t),x(t))\in \mathbb{R}_+^{n+1}$ be a solution of \eqref{eq:model}. Then, for all $t\geq 0$, the solution is bounded,  $s_i(t)>0, \ i=1,2,\dots,n,$ and $x(t)>0$.  
\end{corollary}
\begin{proof}
That solutions to system \eqref{eq:model} are bounded was part of the proof of \cref{lemma:beating}. It is also clear that the impulse map leaves solutions positive. 
\end{proof}

\section{The Periodic Solution}\label{Sec:Periodic}
Define the component-wise Lyapunov-like function by 
\begin{equation}\label{eq:Lyapunov}
    V_i(\vec{s}) =(s_1^\IN-s_1)y_1-(s_i^\IN-s_i)y_i,\quad i=1,...,n.
\end{equation}
Each component, $V_i(\vec{s})$, can be seen as the signed distance from $\vec{s}$ to the line through $\vec{s}^\IN$ in the direction of $\vec{Y}$ when both are projected onto the $s_1$-$s_i$ plane. If $V_i(\vec{s})>0$, then $\vec{s}$ lies above the line through $\vec{s}^\IN$ in the $s_1$-$s_i$ plane, and if $V_i(\vec{s})<0$, then $\vec{s}$ lies below the line through $\vec{s}^\IN$ in the $s_1$-$s_i$ plane.  Note that $V_1(\vec{s})\equiv0$ and if $n=2$, then $V_2(\vec{s})$ is the same Lyapunov-type function used in \cite{Hsu2019}.

While each $V_i(\vec{s})$ is useful to determine the location of the projection of $\vec{s}$ in the $s_1$-$s_i$ plane, they are not convex functions, and therefore $\vec{V}(\vec{s})$ does not truly constitute a vector-Lyapunov function. On the other hand, the supremum norm, 
\begin{equation}
    \Vert \vec{V}(\vec{s})\Vert_\infty = \max \{\vert V_i(\vec{s})\vert : i = 1,\dots,n\},
\end{equation}
is convex and is therefore a candidate Lyapunov function.

\begin{lemma}\label{lemma:Lyapunov} 
Assume that $(s_1(t),\dots,s_n(t),x(t)) \in \mathbb{R}_+^{n+1}$ is a solution of \eqref{eq:model}. Let $t_0=0$ and $t_k$ be the $k$th impulse time, if it exists. Otherwise, set $t_k=\infty$. Then, for each $i=2,\dots, n$,
\begin{enumerate}
    \item $\frac{d}{dt}V_i(\vec{s}(t))=0$ for $t\in (t_k,t_{k+1})$. 
    \item $V_i(\vec{s}(t_{k}^+)) =(1-r) V_i(\vec{s}(t_k^-))$. 
\end{enumerate}
\end{lemma}
\begin{proof}
For each component of $\vec{V}$, 
\begin{align*}
    \frac{d}{dt}V_i(\vec{s}(t)) &= \frac{d}{dt}y_1(s_1^\IN-s_1(t)) - \frac{d}{dt}y_i(s_i^\IN-s_i(t)),\\
    &= -F(\vec{s}(t))x(t)+F(\vec{s}(t))x(t),\\ &= 0,
\end{align*}
and so $\frac{d}{dt}\max\{\vert V_i(\vec{s}(t))\vert:i=1,\dots,n\} = 0$.

When $t = t_k^+$, using \cref{eq:Impulses},
\begin{align*}
    V_i(\vec{s}(t_k^+)) &= y_1(s_1^\IN-s_1(t_k^+))-y_i(s_i^\IN-s_i(t_k^+)),\\
    &=y_1(s_1^\IN-rs_1^\IN-(1-r)s_1(t_k^-))-y_i(s_i^\IN-rs_i^\IN-(1-r) s_i(t_k^-)),\\
    &=(1-r)V_i(\vec{s}(t_k^-)).\qedhere
\end{align*}
\end{proof}

\begin{corollary}\label{corollary:periodic}
If $(s_1(t),\dots,s_n(t),x(t))\in\mathbb{R}_+^{n+1}$ is a solution to \cref{eq:model} with an infinite number of impulses, then $\vec{V}(\vec{s}(t)) \to \vec{V}(\vec{s}^\IN) = \vec{0}$ as $t\to \infty$. 
\end{corollary}
\change{
\begin{proof}
From \cref{lemma:Lyapunov}, it follows that $V_i(\vec{s}(t_k^+))=(1-r)^kV(\vec{s}(t_0))$. Since $(1-r)<1$, $(1-r)^k\to0$ as $k\to\infty$, and thus each component of $\vec{V}(\vec{s})$ converges to $0$ as $t\to\infty$.
\end{proof}}
We can use the components of $\vec{V}(\vec{s})$ to partition $\mathbb{R}^n$ into two complementary pieces.  Define
$$\overline{V_i} = y_1(s_1^\IN-\overline{s_1})-y_is_i^\IN,$$ (i.e., $V_i(\vec{s})$ when $s_1=\overline{s_1}$ and $s_i = 0$), and
\begin{align*}
    \Omega_1 &= \{\vec{s} \in \mathbb{R}^n_+ :s_1\geq\overline{s_1},~  V_i(\vec{s}) > \overline{V_i},~\text{for all}~ i = 2,...,n\},\\
    \Omega_0 &= \{\vec{s} \in \mathbb{R}^n_+ :s_1\geq\overline{s_1},~  V_i(\vec{s}) < \overline{V_i},~\text{for at least one}~ i = 2,...,n\}.
\end{align*}

\begin{figure}
    \centering
    \includegraphics[width=0.8\textwidth]{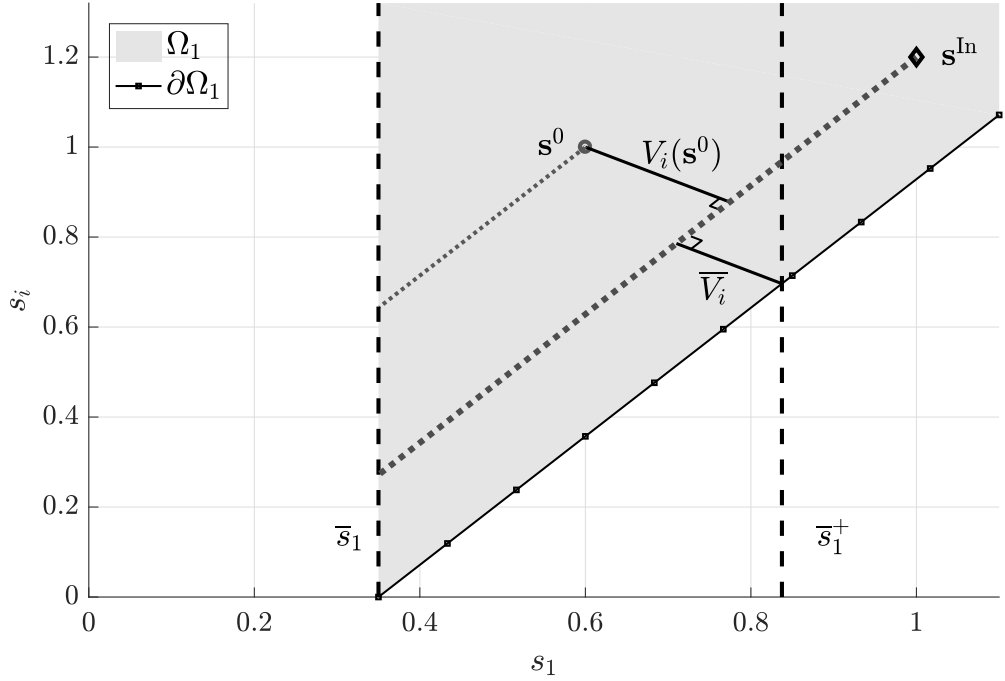}
    \caption{For any $\vec{s}^0$, $V_i(\vec{s}^0)$ is the length of the perpendicular line segment connecting $\vec{s}^0$ to the solution segment through $\vec{s}^\IN$ in the $s_1$-$s_i$ plane. For each $i$, $\overline{V_i}$ is the distance from $\partial\Omega_1$ to $\vec{s}^\IN$ in the $s_1$-$s_i$ plane.}
    \label{fig:Lyapunov2}
\end{figure}

\begin{lemma}\label{lemma:Omeganot} 
If $(s_1(t),\dots,s_n(t),x(t))\in \mathbb{R}_+^{n+1}$ is a solution of \eqref{eq:model} with $\vec{s}(0)\in \Omega_0$, then there are no impulses. 
\end{lemma}
\begin{proof}
Without loss of generality, assume that $V_2(\vec{s}^0)<\overline{V}_2$. Suppose that the first impulse occurs at $t=t_1$; i.e., $s_1(t_1^-) = \overline{s_1}$. \change{By the first property of \cref{lemma:Lyapunov},} 
$$
    y_1(s_1^\IN-\overline{s_1})-y_2(s_2^\IN-s_2(t_1^-))=V_2(\vec{s}(t_1^-))=V_2(\vec{s}^0)<\overline{V}_2= y_1(s_1^\IN-\overline{s_1})-y_2s_2^\IN.
$$
 This implies $s_2(t_1^-)<0$, contradicting \cref{lemma:Bio_well_posed}, and so there are no impulses.
\end{proof}

\begin{lemma}\label{lemma:fail} 
If $\vec{s}^\IN \in \Omega_0$, then there are at most a finite number of impulses and $\lim_{t\to\infty} x(t) = 0$. 
\end{lemma}
\begin{proof}
Suppose not. Then there exists an infinite sequence of impulse times $\{t_k\}_{k=1}^{\infty}$. Since $\vec{s}^{\IN}\in\Omega_0$, it follows that $V_i(\vec{s}^\IN)=0<\overline{V_i}$ for at least one $i=2,...,n$. By \cref{corollary:periodic}, there exists $k\geq 0$ such that $V_i(\varphi_0(\vec{s}^k))<\overline{V_i}$. Therefore, $\varphi_0(\vec{s}^k)\in \Omega_0$, and by \cref{lemma:Omeganot}, no more impulses can occur. Thus, the remaining dynamics are governed by \cref{eq:ODES}. By \cref{lemma:odes},  $x(t)\to0$ as $t\to \infty$.
\end{proof}

\begin{rmk}
\change{Neither $\Omega_1$ nor $\Omega_0$ are closed sets in the subspace topology on $\{\vec{s}\in\mathbb{R}^n_+: s_1\geq \overline{s_1}\}$, which is the subset of $\mathbb{R}_n$ reachable by solutions. These sets are complementary in the sense that $\Omega_0\cup \overline{\Omega_1}=\{\vec{s}\in\mathbb{R}^n_+: s_1\geq \overline{s_1}\}$, and $\Omega_0\cap \overline{\Omega_1}=\varnothing$.} We are therefore missing the marginal case on their shared boundary, \begin{align*}
    \partial\Omega_1 &= \{\vec{s}\in\mathbb{R}_+^n: s_1\geq \overline{s_1}, ~ V_i(\vec{s})\geq \overline{V_i},~\text{for all}~i=2,...,n,\\&\quad\text{and} ~V_i(\vec{s})=\overline{V_i}~\text{for at least one}~i=2,\dots,n\}.
    \end{align*}
    While not covered here, it can be seen that if $\vec{s}^0\in\partial\Omega_1$, then there are no impulses. If $\vec{s}^\IN \in \partial\Omega_1$ and $\vec{s}^0\in\Omega_1$, then either finitely many impulses occur or there are infinitely many impulses but the time between impulses tends to infinity. 
\end{rmk}

In order to visualize solutions, we project them onto the $s_1$-$s_j$ plane, where $j$ is such that $\overline{V_j}=\max \{\overline{V_i}:i=2,\dots,n\}$ . This allows us to see clearly whether $\vec{s}^\IN \in \Omega_0$ or  $\vec{s}^\IN\in\Omega_1$, since if $\vec{s}^\IN\in\Omega_0$, then at least one $\overline{V_i}>0$.
\begin{example}\label{example1}
Consider \eqref{eq:model} with $n=3$, 
\begin{equation*}
    F(\vec{s}) = \min\left\{\frac{0.4s_1}{0.25+s_1},\frac{1.3s_2}{0.3+s_2},\frac{0.5s_3}{0.5+s_3}\right\},
\end{equation*}
\sloppy $r = 0.7$, $\vec{Y}=(1.00,0.83,1.25)^T$,
         $\overline{s_1}=0.4$,
         $D = 0.05$ and $\vec{s}^\IN=(1,1,0.6)^T$. Using its definition, we compute $\overline{\vec{V}} = (0,-0.20,0.52)^T$. Since $\overline{V_3}=\max\{\overline{V_i}:i=2,3\}$, we project solutions onto the $s_1$-$s_3$ plane and easily see that $\vec{s}^\IN\in\Omega_0$.  The initial conditions, $\vec{s}^0=(0.6,0.7,0.8)^T,~x^0 = 0.5$ satisfy $\vec{s}^0\in\Omega_1$, yet the conditions for \cref{lemma:Omeganot} are satisfied, and so, as predicted, in \cref{fig:Omega0}, we see  that $x(t)\to 0$ as $t\to\infty$.
\begin{figure}[h]
    \centering
    \includegraphics[width=\textwidth]{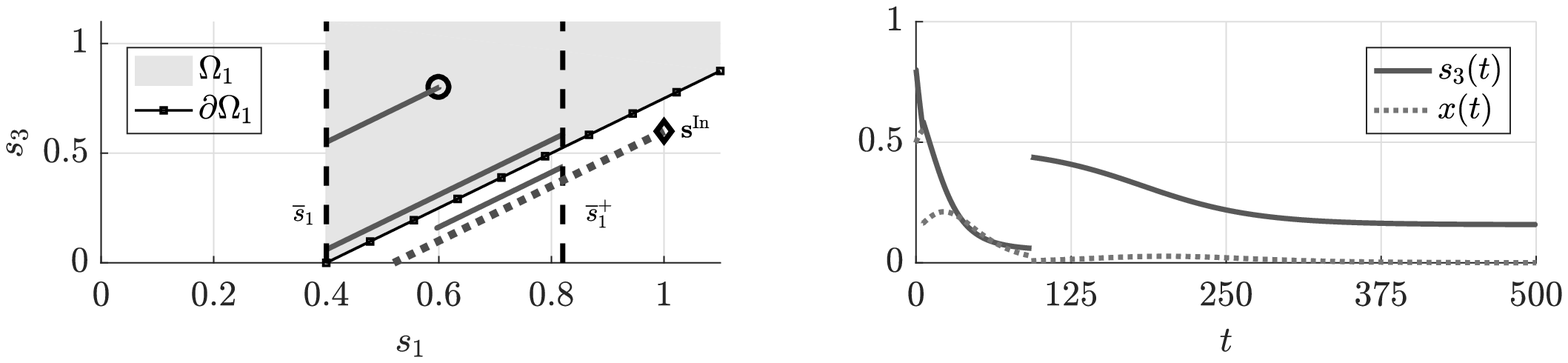}
    \caption{The dynamics of \cref{example1} illustrated by projecting orbits onto $s_1$-$s_3$ space, with the line through $\vec{s}^\IN$ shown in dotted red on the left. Solutions of $s_3$ and $x$ as functions of time are shown on the right. As predicted by \cref{lemma:Omeganot},  only finitely many impulses occur and $x(t)\to 0$ as $t\to\infty$.}
    \label{fig:Omega0}
\end{figure}
 
\end{example}

If $\vec{s}^\IN\in\Omega_1$, then each component of $\varphi_1(\vec{s}^\IN)$ is positive. We define $\widehat{\vec{s}}^+$ to be the point on $\varphi_\nu(\vec{s}^\IN)$ with $s_1=\overline{s_1}^+$, i.e., for fixed $r\in(0,1)$
\begin{equation*}
    \widehat{\vec{s}}^+:=\widehat{\vec{s}}^+(r) = \vec{s}^\IN-(1-r)y_1(s_1^\IN-\overline{s_1})\vec{Y}\change{= \varphi_{(1-r)}(\vec{s}^\IN)},
\end{equation*}
and define 
\begin{equation}\label{eq:mu_alt}
    \mu(r) = y_1(\overline{s_1}^+-\overline{s_1})\int_0^1\left(1-\frac{D}{F(\varphi_\nu(\widehat{\vec{s}}^+))}\right)d\nu
\end{equation}
to be the change in $x$ as $\vec{s}$ changes from $\widehat{\vec{s}}^+$ to $\widehat{\vec{s}}=\varphi_1(\widehat{\vec{s}}^+)$. Note that \change{since $\widehat{\vec{s}}^+ =\varphi_{(1-r)}(\vec{s}^\IN)$, $\widehat{\vec{s}}^+$ and $\vec{s}^\IN$ lie on the same solution segment.} Thus, by \cref{lemma:Lyapunov}, $V_i(\widehat{\vec{s}}^+)=V_i(\vec{s}^\IN) = 0$ for all $i=1,...,n$ and for all $r\in(0,1)$. Since $\overline{s_1}^+ = rs_1^\IN +(1-r)\overline{s_1}$, an equivalent representation of \eqref{eq:mu_alt} is
\begin{equation}\label{eq:mu}
    \mu(r) = ry_1(s_1^\IN-\overline{s_1})\int_0^1\left(1-\frac{D}{F(\varphi_\nu(\widehat{\vec{s}}^+))}\right)d\nu.
\end{equation}

\begin{theorem}\label{theorem:Periodic} 
Assume $\vec{s}^\IN \in \Omega_1$. If $r\in(0,1)$ and $\mu(r)>0$, then system \eqref{eq:model} has a unique periodic solution that has one impulse per period. On a periodic solution, $x(t_k^+) = \frac{(1-r)}{r}\mu(r)$ and $x(t_k^-) = \frac{1}{r}\mu(r)$ for all $k\in \mathbb{N}$.\\
\indent If $\mu(r)\leq 0$, then system \eqref{eq:model} has no periodic solutions.
\end{theorem}
\begin{proof}
First we show that if \cref{eq:model} has a periodic solution, then it is unique.

Assume that \cref{eq:model} has a periodic solution. From \cref{corollary:periodic}, the projection of the periodic solution onto the resource hyperplane has to lie on $\varphi_\nu(\widehat{\vec{s}}^+)$.  Since system \eqref{eq:ODES} has no cycles, there is at least one impulse, and, by periodicity, there are an infinite number of impulses. Denote by $K$ the number of impulses in each period. Then $u_\nu(\vec{s}^{K+k},x^{K+k})=u_\nu(\vec{s}^k,x^k)$ for every $\nu\in[0,1]$, $k\in\mathbb{N}$. By \eqref{eq:Impulses} and combining \eqref{eq:x} with \eqref{eq:mu_alt},  
\begin{align*}
    u_1(\vec{s}^k,x^k) = u_0(\vec{s}^k,x^k)+\mu(r), && x^{k+1}=(1-r)u_1(\vec{s}^k,x^k),
\end{align*}
and therefore, using the relation $u_0(\vec{s}^k,x^k) = x^k$,
\begin{align*}
 x^{k+1} = (1-r)(x^k +\mu(r)).
\end{align*}
If $x^{k+1}>x^k$, then we can show inductively that $\{x^k\}_{k=0}^\infty$ is a strictly increasing sequence. Similarly, if $x^{k+1}<x^k$ we can show that $\{x^k\}_{k=0}^\infty$ is a strictly decreasing sequence. Therefore, if there is a periodic orbit, it is unique up to time translation and satisfies $K=1$, $u_0(\vec{s}^k,x^k) = x^k = \frac{1-r}{r}\mu(r)$, and $u_1(\vec{s}^k,x^k)=\frac{1}{r}\mu(r)$ for all $k\in\mathbb{N}$.

If $\vec{s}^\IN\in\Omega_1$ and $\mu(r)>0$, then the solution with  $(\vec{s}^0,x^0)=(\widehat{\vec{s}}^+,\frac{1-r}{r}\mu(r))$ is periodic, since $\varphi_1(\widehat{\vec{s}}^+) =\widehat{\vec{s}}$ and $u_1\left(\widehat{\vec{s}}^+,\frac{1-r}{r}\mu(r)\right)=\frac{1}{r}\mu(r)$.

If $\mu(r)\leq0$, then by the uniqueness of periodic solutions and \cref{lemma:Bio_well_posed}, \cref{eq:model} has no periodic solutions.
\end{proof}

\begin{proposition}\label{proposition:rstar} 
If $\mu(1)>0$, then there exists a unique $r^*\in[0,1)$ such that $\mu(r)>0$ for all $r\in(r^*,1]$ and $\mu(r)\leq0$ for all $r\in[0,r^*]$. 
\end{proposition}
\begin{proof}
Let \begin{equation}\label{eq:rstar}
    r_* = \max\{ r\in[0,1]: \mu(\tau) \leq 0~\text{for all}~ \tau\in[0,r]\}. 
\end{equation}
Note that $r_*$ is well defined, since $\mu(0)=0$ and $\mu$ is a continuous function of $r$. Since $\mu(1)>0$, it follows that $r_*\in[0,1)$. By definition of $r_*$, there exists $\varepsilon>0$ such that
\begin{equation*}
\mu(r)>\mu(r_*)=0
\end{equation*}
for all $r\in(r_*,r_*+\varepsilon)$. If not, then $r_*$ could be increased, violating the definition of $r_*$. For each $\nu\in[0,1]$, $F(\varphi_\nu(\widehat{\vec{s}}^+(r)))$ is a nondecreasing function of $r$, since 
 \begin{align*}\varphi_\nu(\widehat{\vec{s}}^+(r)) &= \widehat{\vec{s}}^+(r)-\nu y_1(\change{\overline{s_1}^+}-\overline{s_1})\vec{Y},\\
 &= \vec{s}^\IN - y_1(s_1^\IN-\overline{s_1})\vec{Y}+r(1-\nu)y_1(s_1^\IN-\overline{s_1})\vec{Y},
 \end{align*}
\change{which follows from $\widehat{s}_1^+=\overline{s_1}^+=rs_1^\IN+(1-r)\overline{s_1}$. } It follows that $\mu(r)>\mu(r_*)$ for all $r\in(r_*,1]$.
\end{proof}

\begin{proposition}\label{proposition:fail}
Assume $\vec{s}^\IN\in\Omega_1$ and let $(s_1(t),\dots,s_n(t),x(t))$ be a solution to \eqref{eq:model} with positive initial conditions.
\begin{enumerate}[label=(\roman*)]
    \item If $\mu(r)<0$, then there are finitely many impulses.
    \item If $\mu(r)=0$, then either finitely many impulses occur or the time between impulses tends to infinity.
\end{enumerate}
\end{proposition}
\begin{proof}
    Suppose the solution has infinitely many impulses. By \cref{corollary:periodic}, $\vec{s}^k\to\vec{\widehat{s}}^+$ as $k\to\infty,$ and by \cref{lemma:Bio_well_posed}, $x^k\geq0$ for all $k\geq0$.  
    
    \textit{(i)} If $\mu(r)<0$, then \change{ 
    \begin{align*}
       x^{k+1}-x^k\leq x^{k+1}-(1-r)x^k=(1-r)y_1(\overline{s_1}^+-\overline{s_1})\int_0^1\left(1-\frac{D}{F(\varphi_\nu(\vec{s}^k))}\right)d\nu.
    \end{align*}
    Note that, since $F$ is Lipschitz continuous, there exists $K>0$ such that
    \begin{align*}
    \sup_{\nu\in[0,1]}\left\vert\frac{D}{F(\varphi_\nu(\vec{s}^k))}-\frac{D}{F(\varphi_\nu(\widehat{\vec{s}}^+))}\right\vert\leq DK \sup_{\nu\in[0,1]}\left\vert \frac{\varphi_\nu(\vec{s}^k)-\varphi_\nu(\widehat{\vec{s}}^+)}{F(\varphi_\nu(\vec{s}^k))F(\varphi_\nu(\widehat{\vec{s}}^+))}\right\vert,
    \end{align*}
    which, since $\vec{s}^k\to\widehat{\vec{s}}^+$, converges to zero. Thus, the integrand converges uniformly as $k\to\infty$ and  
    \begin{equation*}
    (1-r)y_1(\overline{s_1}^+-\overline{s_1})\int_0^1\left(1-\frac{D}{F(\varphi_\nu(\vec{s}^k))}\right)d\nu\to (1-r)\mu(r)<0
    \end{equation*}
    as $k\to \infty$. Thus, there exists $M>0$ such that 
    $x^{k+1}-x^k<\frac{1-r}{2}\mu(r)<0$ for all $k>M$, and therefore $x^k\to -\infty$ as $k\to \infty,$ contradicting \cref{lemma:Bio_well_posed}.}
    
    \textit{(ii)} If $\mu(r)=0$, then 
    \begin{equation*}
        \lim_{k\to\infty} x^{k+1}-(1-r)x^k = 0,
    \end{equation*}
    implying that $x^k\to 0$ as $k\to\infty$. Using the relation $x^{k+1}= (1-r)u_1(\vec{s}^k,x^k)$, it follows that $u_1(\vec{s}^k,x^k)\to 0$ as $k\to\infty$. Therefore, $u_\nu(\vec{s}^k,x^k)$ converges to the heteroclinic orbit of \eqref{eq:ODES} that connects $(\widehat{\vec{s}}^+,0)$ to $(\widehat{\vec{s}},0)$ as $k\to \infty$. This implies that $t_{k+1}-t_k\to \infty$.
\end{proof}

\begin{example}\label{ex:munegative}
Consider \eqref{eq:model} with $n=3$,
\begin{equation*}
    F(\vec{s}) =\frac{0.4s_1}{0.25+s_1}\cdot\frac{1.3s_2}{0.3+s_2}\cdot\frac{0.5s_3}{0.5+s_3},
\end{equation*}
$r=0.7$, $\vec{Y} =(1.00,0.83,1.25)$, $\overline{s_1}=0.4$, $D =0.1$ and $\vec{s}^{\rm{In}}=(1,1,1)$. By definition $\overline{V}_2 = -0.6$ and $\overline{V}_3 = -0.2 $. Since $\overline{V}_3=\max\{\overline{V_2},\overline{V_3}\}$, we project solutions onto the $s_1$--$s_3$ plane, and see that $\vec{s}^\IN\in\Omega_1$. Since $\mu(r)\approx -0.2924<0$, by \cref{proposition:fail}, there are a finite number of impulses and $x(t) \to 0$ as $t\to\infty$. This is illustrated in \cref{fig:munegative}.

\begin{figure}[h]
    \centering
    \includegraphics[width=\textwidth]{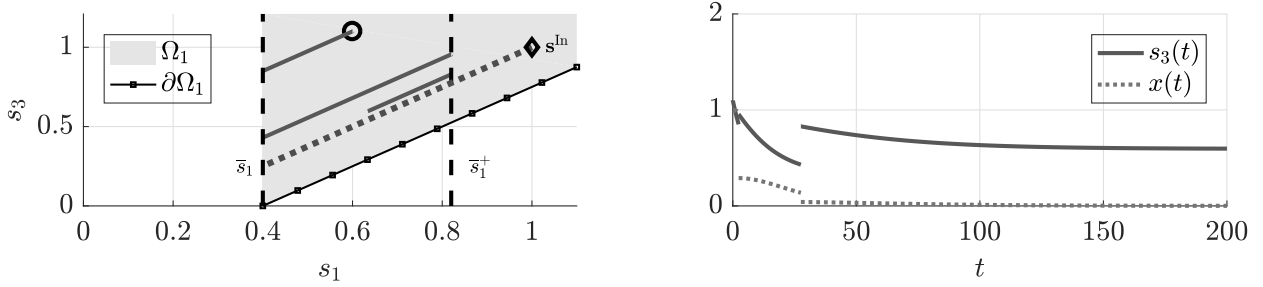}
    \caption{The dynamics of \cref{ex:munegative}, in which $\mu(r)<0$, illustrated by projecting orbits onto $s_1$--$s_3$ space, with the line through $\vec{s}^\IN$ shown in dotted red on the left. Solutions of $s_3$ and $x$ as functions of time are shown on the right. As predicted by \cref{proposition:fail},  only finitely many impulses occur and $x(t)\to 0$ as $t\to\infty$.}
    \label{fig:munegative}
\end{figure}
\end{example}

\subsection{Stability of the Periodic Solution}
In this section, we assume that $\vec{s}^\IN\in\Omega_1$ and $\mu(1)>0$. We fix $r\in(r_*,1)$, where $r_*$ is given in \cref{proposition:rstar}, so that $\mu(r)>0$ and system \eqref{eq:model} has a unique periodic solution. 

For any $\vec{s}^0\in \Omega_1$, we define the net change in $x$ over the time until the first impulse by 
\begin{equation}
    I(\vec{s}^0) =  y_1(s^0_1-\overline{s_1})\int_0^1\left( 1-\frac{D}{F(\varphi_\nu(\vec{s}^0))}\right)d\nu.
\end{equation}
 Since $\vec{s}^0\in\Omega_1$, $I(\vec{s}^0)$ is finite and an impulse occurs as long as $x^0$ is large enough. Note that $I(\widehat{\vec{s}}^+)=\mu(r)$. Define
 \begin{equation}
     \Gamma^+ = \{ \vec{s}\in\mathbb{R}_+^n: s_1 = \overline{s_1}^+\}
 \end{equation}
 and
 \begin{equation}
     G^+=\{\vec{s}\in\Gamma^+\cap\Omega_1:I(\vec{s})>0\},
 \end{equation}
 the subset of $\Gamma^+$ with positive growth before the first impulse.  Also define
\begin{equation}
    G^-=\{\varphi_1(\vec{s})\in\Gamma^-:\vec{s}\in G^+\}
\end{equation}
 the image of $G^+$ under $\varphi_1$ in $\Gamma^-$. Let $g:\Gamma^-\to\Gamma^+$ be the impulse map acting on $\vec{s}$. I.e., for $\vec{s}\in\Gamma^-$,
\begin{equation*}
    g(\vec{s}) =r\vec{s}^\IN+(1-r)\vec{s}. 
\end{equation*}
The composition $(g\circ\varphi_1)(\vec{s}^0)=\vec{s}^{1}$, and more generally $(g\circ\varphi_1)(\vec{s}^k)=\vec{s}^{k+1}$ for $k=0,1,\dots$.

\begin{lemma}\label{lemma:Gamma}
Assume that $\vec{s}^\IN\in\Omega_1$ and $\mu(r)>0$. Then there exists $\rho>0$ such that $\Gamma_\rho^+ := \{\vec{s}\in\Gamma^+:V_i(\vec{s})>-\rho~\text{for all}~i=2,\dots,n\}$ is a subset of $G^+$.
\end{lemma}
\begin{proof}
Let $\widetilde{\vec{s}}(z)=\widehat{\vec{s}}^+-(0,z/y_2,\dots,z/y_n)^T$. Then, by \cref{lemma:Lyapunov},
\begin{align*}
    V_i(\widetilde{\vec{s}}(z)) = y_1\bigg(s_1^\IN-\overline{s_1}^+\bigg)-y_i\left(s^\IN_i-\bigg(\widehat{s_i}^+-\frac{z}{y_i}\bigg)\right)=V_i(\widehat{\vec{s}}^+)-z=-z
\end{align*}
for $i=2,\dots,n$. Since $\vec{s}^\IN\in\Omega_1$, $\overline{V_i}<0$ for all $i=2,\dots,n$. Let  $ \sigma =\min\{-\overline{V_i}: i=2,\dots,n\} > 0$. \change{Then $\widetilde{\vec{s}}(\sigma)\in\partial \Omega_1$, and thus, by \cref{lemma:Lyapunov},} $\varphi_\nu(\widetilde{\vec{s}}(\sigma))$ is in $\partial\Omega_1\subset\mathbb{R}^n_+$ for all $\nu\in[0,1)$ and, \change{by the definition of $\overline{V_i}$, $\varphi_\nu(\widetilde{\vec{s}}(\sigma))$} intersects the boundary of $\mathbb{R}^n_+$ when $\nu =1$. Thus, $F(\varphi_\nu(\widetilde{\vec{s}}(\sigma)))>0$ for all $\nu\in[0,1)$ and
$F(\varphi_1(\widetilde{\vec{s}}(\sigma)))=0$.
\change{Since $F(\varphi_\nu(\vec{s}))$ is Lipschitz continuous and decreasing in $\nu$, there exists $K>0$ such that
$$ F(\varphi_\nu(\vec{s}))-F(\varphi_1(\vec{s})) \leq K( 1-\nu),$$ for all $\nu\in[0,1]$.  Since $F$ is increasing in each of its components,  $F(\varphi_1(\widetilde{\vec{s}}(\sigma-\delta)))>F(\varphi_1(\widetilde{\vec{s}}(\sigma)))=0$ for all $\delta>0$. By continuity, there exists $\delta>0$ sufficiently small such that $0<F(\varphi_1(\widetilde{\vec{s}}(\sigma-\delta)))\leq Ke^{-\frac{K}{D}}$. Thus 
\begin{align*}
    I(\widetilde{\vec{s}}(\sigma-\delta)) &= y_1(\overline{s_1}^+-\overline{s_1})\int_0^1 \left(1-\frac{D}{F(\varphi_\nu(\widetilde{\vec{s}}(\sigma-\delta)))} \right)d\nu\\
    &=y_1(\overline{s_1}^+-\overline{s_1})\int_0^1 \left(1-\frac{D}{F(\varphi_\nu(\widetilde{\vec{s}}(\sigma-\delta)))-F(\varphi_1(\widetilde{\vec{s}}(\sigma-\delta)))+F(\varphi_1(\widetilde{\vec{s}}(\sigma-\delta)))} \right)d\nu\\
    &\leq y_1(\overline{s_1}^+-\overline{s_1})\int_0^1 \left(1-\frac{D}{K(1-\nu)+Ke^{-\frac{K}{D}}} \right)d\nu\\ 
    &= -y_1(\overline{s_1}^+-\overline{s_1})\frac{D}{K}\log\left(1+e^{-\frac{K}{D}}\right)<0.
\end{align*}}

For $z<\sigma$, $\widetilde{\vec{s}}(z)\in\Omega_1$ and $I(\widetilde{\vec{s}}(z))$ is a continuous function of $z$. Since  $I(\widetilde{\vec{s}}(0))=I(\widehat{\vec{s}}^+) = \mu(r)>0$, by the intermediate-value theorem there exists $z\in (0,\sigma)$ such that $I(\widetilde{\vec{s}}(z))=0$. Let $\rho = \sup\{z\in(0,\sigma): I(\widetilde{\vec{s}}(z))>0\}$. Thus, the set $\Gamma_\rho^+$ is well defined, and all that is left is to show that $\Gamma_\rho^+\subset G^+$. 

Let $\vec{s}\in\Gamma_\rho^+$. Then there exists $\varepsilon>0$ such that $V_i(\vec{s})>-\rho+\varepsilon = V_i(\widetilde{\vec{s}}(\rho-\varepsilon))$ for each $i=2,\dots,n$. This implies that \change{$s_i>\widehat{s_i}^+-(\rho-\varepsilon)/y_i = \widetilde{s_i}(\rho-\varepsilon)$, i.e., that each component of $\vec{s}$ is larger than the corresponding component of $\widetilde{\vec{s}}(\rho-\varepsilon)$.} By the definition of $\rho$, we have $I(\widetilde{\vec{s}}(\rho-\varepsilon))>0$. Since $F(\vec{s})$ is nondecreasing in each of the $s_i$,
\begin{equation*}    I(\vec{s})\geq I(\widetilde{\vec{s}}(\rho-\varepsilon))>0. \qedhere \end{equation*}
\end{proof}

If $n=2$, then \cref{lemma:Gamma} implies that there exists $s_2^\flat>0$ such that $G^-=\{\overline{s_1}\}\times(s_2^\flat,\infty)$. This is the result of Lemma 4.9 in \cite{Hsu2019}. If $n>2$, then we are unable to find such an explicit formulation of \change{$G^-$.}

We use the set $G^-$ to define
\begin{equation}
    \Omega_{G} = \{\vec{s}^0\in\Omega_1: \varphi_1(\vec{s}^0)\in G^-\},
\end{equation}
 the set of points in $\Omega_1$ that will flow through $G^-$ for some value of $x^0$. Using \eqref{eq:Lyapunov} and \cref{lemma:Gamma}, we define
\begin{align*}
    \change{ \Gamma_\rho^- &=\{ \vec{s}\in\Gamma^-:V_i(\vec{s}) >-\rho, i = 2,\dots,n\}}\\
    \Omega^\rho &= \{\vec{s}\in\Omega_1: V_i(\vec{s}) >-\rho, i = 2,\dots,n\},
\end{align*}
where $\rho$ is given in \cref{lemma:Gamma}. It is clear that
\change{$\Gamma^-_\rho\subset G^-$ and $\Omega^\rho\subseteq    \Omega_G$.}
\change{\begin{rmk}
The set $\Gamma^+_\rho$ is convex since if $\vec{p}\in\Gamma^+_\rho$ and $\vec{q}\in\Gamma^+_\rho$, then 
\begin{align*}
    V_i(\tau\vec{p}+(1-\tau)\vec{q}) & = y_1(s_1^\IN-\overline{s_1}^+)-y_i(s_i^\IN-\tau p_i-(1-\tau)q_i),\\
    &=\tau y_1(s_1^\IN-\overline{s_1}^+)-\tau y_i(s_i^\IN-p_i)+(1-\tau)y_1(s_1^\IN-\overline{s_1}^+)-(1-\tau)y_i(s_i^\IN-q_i),\\
    &=\tau V_i(\vec{p})+(1-\tau)V_i(\vec{q}),\\
    &>-\tau \rho -(1-\tau)\rho = -\rho,
\end{align*}
for all $\tau\in[0,1]$. In particular, if $\vec{s}^k\in\Gamma^+_\rho$, then \begin{align*}
\vec{s}^{k+1} &= r\vec{s}^\IN +(1-r)\varphi_1(\vec{s}^k),\\
&= r\vec{s}^\IN -(1-r)y_1(rs_1^\IN-r\overline{s_1})\vec{Y}+(1-r)\vec{s}^k,\\
&= r\widehat{\vec{s}}^++(1-r)\vec{s}^k.
\end{align*} 
Thus, $\vec{s}^{k+1}$ is a convex combination of two points in $\Gamma^+_\rho$, and therefore an element of $\Gamma^+_\rho$ itself.  This also implies that if $\vec{s}\in\Gamma^-_\rho$, then $g(\vec{s})\in\Gamma^+_\rho$. \\
In general, the set $G^+$ might not be convex unless we impose further restrictions on $F$, and so it may not be true for all functions $F$ that if $\vec{s}\in G^-$, then $g(\vec{s})\in G^+.$
\end{rmk}}

\begin{lemma}\label{lemma:convergence}Assume that $\vec{s}^\IN\in \Omega_1$ and $\mu(r)>0$. Let $(s_1(t),\dots,s_n(t),x(t))$ be a solution of system \eqref{eq:model} with $x^0>0$ and \change{$\vec{s}^0\in\Omega^\rho$.} 
\begin{enumerate}
    \item If $x^0\leq -I(\vec{s}^0)$, then there are no impulses. 
    \item As $t\to \infty$, $(s_1(t),\dots,s_n(t),x(t))$ converges to the unique periodic orbit given by \cref{theorem:Periodic} if and only if $x^0>-I(\vec{s}^0)$.
\end{enumerate}
\end{lemma}
\begin{proof}
Suppose $x^0\leq-I(\vec{s}^0)$ and there is at least one impulse. By \cref{eq:x} and the definition of $I(\vec{s}^0),$
\begin{equation*}
    u_1(\vec{s}^0,x^0) = x^0 + I(\vec{s}^0) \leq 0.
\end{equation*}
This implies that $x(t)=0$ for some finite value of $t$, contradicting the uniqueness of initial values problems to ODEs. 

If $x^0>-I(\vec{s}^0)$, then by \cref{eq:x} and the definition of $I(\vec{s}^0)$, at least one impulse occurs. Let $t=t_1^-$ be the time of the first impulse. Since $\vec{s}^0\in\change{\Omega^\rho}$, we have $\vec{s}(t_1^-)=\varphi_1(\vec{s}^0)\in \change{\Gamma^-_\rho}$. It follows that $\vec{s}^1=r\vec{s}^\IN+(1-r)\varphi_1(\vec{s}^0)\change{\in \Gamma^+_\rho}$ and thus that $I(\vec{s}^1) > 0$. Therefore, there is a second impulse at $t=t_2^-$. Inductively, it follows that impulses occur indefinitely. By \cref{corollary:periodic}, $\lim_{k\to\infty}\Vert \vec{V}(\varphi_\nu(\vec{s}^k)\Vert_\infty =  0$ for all $\nu\in[0,1]$, and therefore $\vec{s}^k\to \widehat{\vec{s}}^+$ as $t\to \infty$. By \eqref{eq:x} and the relationship $I(\widehat{\vec{s}}^+) = \mu(r),$
\begin{equation*}
\lim_{k\to\infty}(u_1(\vec{s}^k,x^k) -u_0(\vec{s}^k,x^k) )= \mu(r).
\end{equation*}
On the other hand, the impulse map in \cref{eq:Impulses} gives 
\begin{equation*}
    \lim_{k\to\infty}(u_0(\vec{s}^{k+1},x^{k+1}) - (1-r)u_1(\vec{s}^k,x^k) )= 0.
\end{equation*}
Combining these, and using the fact that $u_0(\vec{s}^k,x^k) = x^k$, leads to 
\begin{equation}
    \lim_{k\to\infty}(x^{k+1}-(1-r)x^k) = (1-r)\mu(r).
\end{equation}
This implies that $\lim_{k\to\infty}x^k=\frac{1-r}{r}\mu(r)$ and $\lim_{k\to\infty}u_1(\vec{s}^k,u^k) = \frac{1}{r}\mu(r)$. 
\end{proof}
\begin{corollary}
If $\vec{s}^\IN\in\Omega_1$ and $\mu(r)>0$, then all solutions to \eqref{eq:model} with $x^0>0$ and $\vec{s}^0=\vec{s}^\IN$ converge to the periodic orbit given in \cref{theorem:Periodic}.
\end{corollary}
\begin{proof}
Since $\vec{s}^0=\vec{s}^\IN$, $I(\vec{s}^0)>\mu(r)>0$, and so $x^0 > 0 >-I(\vec{s}^0).$
\end{proof}

For each $\vec{s}^0\in\Omega_1$ let $N_0=N_0(\vec{s}^0)$ be the smallest positive integer such that $\vec{s}^{N_0}\in G^+$. Clearly, if $\vec{s}^0\in\Omega_G$, we have $N_0(\vec{s}^0)=1$. 

In general, we are unable to get an exact characterization of $\Omega_G$ in terms of $\vec{V}(\vec{s}^0)$. However, we can approximate $N_0$ using $\Omega^\rho$. Let $N^\rho$ be the smallest positive integer such that $\vec{s}^{N^\rho}\in\Omega^\rho$. By applying \cref{lemma:Lyapunov} repeatedly,   
\begin{equation}\label{eq:integer}
    V_i(\vec{s}^k) = (1-r)^kV_i(\vec{s}^0).
\end{equation}
The condition that $\vec{s}^0\in\Omega_1\setminus\Omega^\rho$ is equivalent to $V_i(\vec{s}^0)\leq-\rho$ for at least one of $i=2,...,n$. By applying this to \cref{eq:integer} and solving for $k$, 
\begin{equation}\label{Nrho}
    N^\rho =  \max\left\{\left\lceil\frac{\ln(V_i(\vec{s}^0)/-\rho)}{-\ln(1-r)}\right\rceil: V_i(\vec{s}^0)\leq -\rho\right\},
    \end{equation}
    where $\lceil x \rceil$ is least integer greater than $x$. 
By \cref{lemma:Lyapunov} and since $\Omega^\rho\subset\Omega_G$, $ N_0\leq N^\rho$.  From \eqref{Nrho}, we see that $N^\rho$ has the upper bound 
\begin{equation*}
    \overline{N} =\max \left\{ \left\lceil \frac{\ln(\overline{V}_i/-\rho)}{-\ln(1-r)}\right\rceil: i = 2,\dots,n\right\},
\end{equation*}
and so $N_0\leq \overline{N}$; i.e., every trajectory enters $\Omega_G$ after finitely many impulses, or the reactor fails before then.  

For any solution to \eqref{eq:model} with $x^0>0$ and $\vec{s}^0\in\Omega_1$, if there exists $t_1^-$ with $s_1(t_1^-) = \overline{s_1}$,
\begin{equation*}
    x(t_1^-) = u_1(\vec{s}^0,x^0)=x^0+I(\vec{s}^0),
\end{equation*}
and, for any $k=2,3,...$, the value of $x(t_k^-)$ is given by
\begin{equation*}
    x(t_k^-) = x^k +I(\vec{s}^k).
\end{equation*}
Inductively,  
\begin{equation*}
    x(t_k^-)= (1-r)^{k-1}x^0 +\sum_{j=1}^k(1-r)^{k-j}I((g\circ\varphi_1)^{j-1}(\vec{s}^0)),
\end{equation*}
and therefore, $x(t_k^-)>0$ is equivalent to
\begin{equation*}
    x^0 > -\sum_{j=1}^k (1-r)^{1-j}I((g\circ\varphi_1)^{j-1}(\vec{s}^0)).
\end{equation*}
We define $X(\vec{s}^0)$ to be the minimum value of $x^0$ required for $\vec{s}(t_*^-)\in \change{\Gamma^-_\rho}$ for some $t_*^-$,
\begin{equation}\label{eq:X}
    X(\vec{s}^0) = -\min_{1\leq k \leq \change{N^\rho}}\left(\sum_{j=1}^k (1-r)^{1-j}I((g\circ\varphi_1)^{j-1}(\vec{s}^0))\right).
\end{equation}
In particular, if $\vec{s}^0\in\change{\Omega^\rho}$, then $X(\vec{s}^0) = -I(\vec{s}^0),$ since $\change{N^\rho=1.}$ 

\begin{proposition}\label{proposition:convergence}
Assume $\vec{s}^\IN\in\Omega_1$ and $\mu(r)>0$. Let $(s_1(t),\dots,s_n(t),x(t))$ be a solution of \eqref{Sec:Model} with $\vec{s}^0\in\Omega_1$ and $x^0>0$.
\begin{enumerate}[label=(\roman*)]
    \item If $x^0\leq X(\vec{s}^0)$, then there are at most \change{$N^\rho-1$} impulses.
    \item If $x^0>X(\vec{s}^0)$, then the solutions converge to the periodic orbit given in \cref{theorem:Periodic}.
\end{enumerate}
\end{proposition}
\begin{proof}
\textit{(i)} Suppose $x^0\leq X(\vec{s}^0)$ and there are at least $\change{N^\rho}$ impulses. Denote the first \change{$N^\rho$} impulse times by $t_1<t_2<...<t_{\change{N^\rho}}$. By \cref{eq:x} and the definition of $X(\vec{s}^0)$,
\begin{equation*}
    x(t_k^-) = u_1(\vec{s}^{k-1},x^{k-1}) =(1-r)^{k-1}(x^0 - X(\vec{s}^0))\leq 0,
\end{equation*}
for some \change{$k<N^\rho$}, which contradicts \cref{lemma:Bio_well_posed}.

\textit{(ii)} If $x^0>X(\vec{s}^0)$, then the solution has at least \change{$N^\rho$} impulses. Then $\vec{s}^{\change{N^\rho}}=(g\circ\varphi_1)^{\change{N^\rho}}(\vec{s}^0)\in\change{\Omega^\rho}$. Since $\vec{s}^{\change{N^\rho}}\in\change{\Gamma^+_\rho}$, we have $I(\vec{s}^{\change{N^\rho}})>0$, and the result follows from \cref{lemma:convergence}.
\end{proof}

\begin{example}\label{ex:minx}
Consider \eqref{eq:model} with $n=3$, 
\begin{equation*}
    F(\vec{s}) =\min\left\{\frac{0.5s_1}{1+s_1},\frac{0.7s_2}{0.4+s_2},\frac{s_3}{1+s_3}  \right\},
\end{equation*}
and $r=0.3$, $\vec{Y} =(2.0,0.2,1.0)^T$, $\overline{s_1}=0.25$, $D = 0.1$ and $\vec{s}^\IN = (0.5,0.1,0.5)$. 

\sloppy By definition, $\overline{V_2} = -0.375$, and $\overline{V_3}= -0.375$. Therefore, $\overline{V_2} = \overline{V_3} = \max \{ \overline{V_2},\overline{V_3}\}$. We are free to project solutions onto either the $s_1$-$s_2$ plane, or the $s_1$-$s_3$. Notice $\vec{s}^\IN\in\Omega_1$ and $\mu(r) \approx0.0037>0$. By \cref{theorem:Periodic} there exists a periodic solution. With the initial conditions $\vec{s}^0 = (0.3,0.01,1)^T$, we have $V(\vec{s}^0) = (0,-0.35,0.6)^T$, and so $\vec{s}^0\in\Omega_1$. We calculate the sum in \eqref{eq:X} for $n=1,\dots,N_0$
where $N^0$ is the first integer such that $(1-r)^{1-n}I(\vec{s}^{n-1})>0$. The approximate values are as follows:
\begin{center}
    \begin{tabular}{c|c|c|c|c|c|c}
        $n$ & 1 & 2 & 3 & 4 & 5 & 6  \\\hline
        $(1-r)^{1-n}I(\vec{s}^{n-1})$ & $-0.1766$ & $-0.0575$ & $-0.330$ & $-0.206$ & $-0.0104$ & $0.0007$
    \end{tabular}
\end{center}
We therefore calculate $X(\vec{s}^0) \approx 0.1766 + 0.0575 + 0.330 + 0.206 + 0.0104 = 0.2981$.  In \cref{fig:minx} (top) the initial biomass concentration is $x^0 = 0.29 < X(\vec{s}^0)$ and so by \cref{proposition:convergence}, $x(t)\to 0$ after at most 4 impulses. In \cref{fig:minx} (bottom) the initial biomass concentration is $x^0 = 0.31>X(\vec{s}^0)$, and so by \cref{proposition:convergence}, the solution converges to the periodic solution as $t\to\infty.$

\begin{figure}[]
    \centering
    \includegraphics[width=\textwidth]{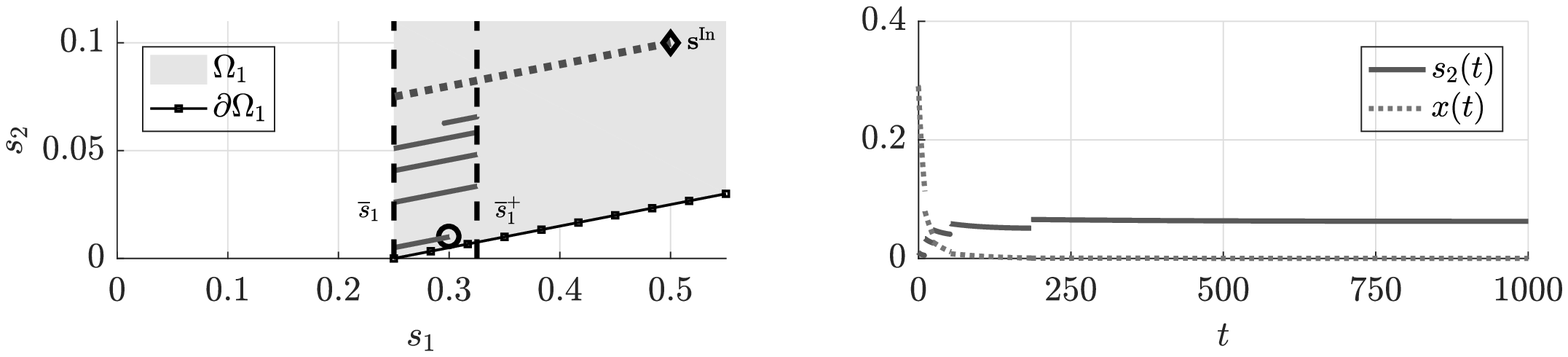}
    \includegraphics[width=\textwidth]{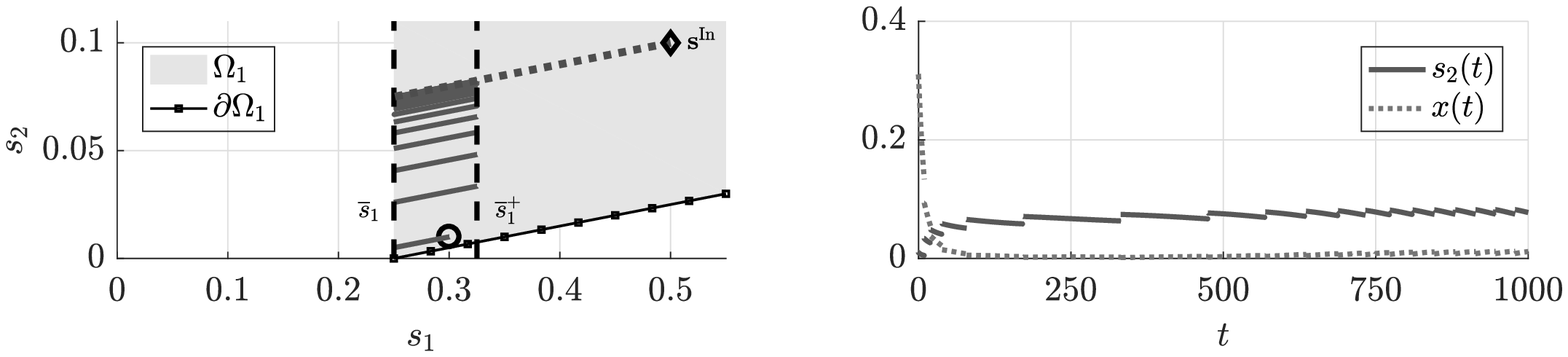}
    \caption{The dynamics of \cref{ex:minx} illustrated by projecting orbits onto $s_1$-$s_2$ space, with the line through $\vec{s}^\IN$ shown in dotted red on the left. Solutions of $s_2$ and $x$ as functions of time are shown on the right. On the top, $x^0< X(\vec{s}^0)$ and so $x(t)\to 0$ as $t\to\infty$ after at most $N_0=4$ impulses. On the bottom $x^0>X(\vec{s}^0)$ and so solutions converge to the periodic solution. }
    \label{fig:minx}
\end{figure}
\end{example}

The following theorem summarizes the results.

\begin{theorem}\label{theorem:Main}
Let $(s_1(t),\dots,s_n(t),x(t))$ be a solution of \eqref{eq:model} with positive initial conditions.
\begin{enumerate}[label=(\roman*)]
    \item If $\vec{s}^\IN\in\Omega_0$, then $(s_1(t),\dots,s_n(t),x(t))$ has only finitely many impulses, and $x(t)\to 0$ as $t\to \infty$.
    \item If $\vec{s}^\IN\in\Omega_1$ and $\mu(r)\leq0$, then $(s_1(t),\dots,s_n(t),x(t))$ either has only finitely many impulses and $x(t)\to0$ as $t\to\infty$ or the time between impulses tends to infinity and $\liminf_{t\to\infty}x(t) = 0$. 
    \item If $\vec{s}^\IN\in\Omega_1$ and $\mu(r)>0$, then there is a unique periodic orbit. Either $(s_1(t),\dots,s_n(t),x(t))$ has infinitely many impulses and converges to the periodic orbit or $(s_1(t),\dots,s_n(t),x(t))$ has only finitely many impulses and $x(t)\to 0$ as $t\to\infty$. The case with infinitely many impulses occurs if and only if 
    \begin{equation*}
        \vec{s}^0\in\Omega_1,\quad \text{and}\quad x^0>X(\vec{s}^0).
    \end{equation*}
    \end{enumerate}
\end{theorem}
    \begin{proof}
        The results follow from \cref{lemma:Omeganot,lemma:fail,theorem:Periodic,proposition:convergence,proposition:fail}.
    \end{proof}

\section{Conclusions}\label{Sec:Conclusions}
We have modelled the self-cycling-fermentation process assuming that there are an arbitrary number of essential resources, $\vec{s}\in\mathbb{R}^n$, that are 
growth limiting for a population of microogranisms, $x$, using a system of impulsive differential equations. We assume that the criterion for decanting the reactor occurs when the concentration of the first nutrient reaches a threshold, $\overline{s_1}$. The process is considered successful if, once initiated, it proceeds indefinitely without intervention. 

By solving the associated system of ODEs in terms of the first nutrient, $s_1$, we have shown that the solutions, when projected onto the nutrient hyperplane, are lines in the direction of $(1/y_1,...,1/y_n)^T$, where $y_i$ is the yield coefficient of the $i$th nutrient. Using a vector Lyapunov function, we divide the nutrient hyperplane into two regions, $\Omega_0$ and $\Omega_1$. The model predicts that if the initial nutrient concentrations lie in $\Omega_0$ then solutions will approach the faces of $\mathbb{R}_+^n$ before $s_1$ reaches $\overline{s_1}$, and the reactor will fail. If the initial nutrient concentrations lie in $\Omega_1$, then the concentration of $s_1$ may reach $\overline{s_1}$, but successful operation of the reactor may still be limited by other factors. 

In reality, we expect that the initial nutrient concentrations are equal to the nutrient concentrations in the input; i.e. $\vec{s}(0) = \vec{s}^\IN$. If, for any solution with initial nutrient concentration $\vec{s}^\IN$ and positive initial biomass concentration ($x(0)>0$), the threshold concentration of $s_1$ is reached with net positive growth of the biomass, then we can pick a fraction of medium to remove, $r$, so that the reactor will cycle indefinitely. In this case, the solutions converge to a periodic solution, with period equal to the length of one cycle. 

If the model has a periodic solution, the nutrient components of the periodic solution lie along the line through $\vec{s}^\IN$ in the direction of $(1/y_1,...,1/y_n)^T$. The net change in biomass along the periodic orbit, denoted $\mu(r)$, must be positive. For other initial nutrient concentrations in $\Omega_1$, the solutions may converge to the periodic solution. However, there is a minimum concentration of biomass, $X$, that is dependent on the initial nutrient concentrations, required for the successful operation of the reactor. If the initial biomass concentration is higher than $X$, then the reactor will cycle indefinitely and solutions will approach the periodic solution. If the initial biomass concentrations are less than $X$, then the reactor will fail after a finite number of cycles. If the model does not have a periodic solution, then the reactor will either fail after a finite number of cycles or it will cycle indefinitely, but the time each cycle takes will grow larger and larger, approaching infinity. 

The model presented here can be thought of as an extension of the single resource model developed in Smith and Wolkowicz \cite{Smith2001}. In that model, it was shown that, when a periodic orbit exists, the reactor will either cycle indefinitely or the reactor will fail without reaching the threshold concentration of $s_1$. We have shown that if there are more essential limiting nutrients but only one is used for the decanting criteria, then the reactor may fail after many cycles, even if the system has a periodic solution. An example of failure after 4 cycles is shown in \cref{fig:minx}. This may offer an explanation for failure of the reactor when the analysis of the single resource model suggests the reactor should operate successfully. 

\section*{Acknowledgements}
The research of Gail S.K. Wolkowicz is supported by the Natural Sciences and Engineering Research Council Discovery Grant \# 9358 and Accelerator supplement.

\bibliographystyle{plain}
\bibliography{main}

\begin{thebibliography}{10}

\bibitem{Bader1978}
F.~Bader.
\newblock Analysis of double-substrate limited growth.
\newblock {\em Biotechnol. and bioeng.}, 20(2):183--202, 1978.

\bibitem{Bauinov1995}
D.D. Ba\u{i}nov and P.S. Simeonov.
\newblock {\em Impulsive differential equations}, volume~28 of {\em Series on
  Advances in Mathematics for Applied Sciences}.
\newblock World Scientific Publishing Co., Inc., River Edge, NJ, 1995.
\newblock Asymptotic properties of the solutions, Translated from the Bulgarian
  manuscript by V. Covachev [V. Khr. Kovachev].

\bibitem{Cordova2014}
F.~{C{\'o}rdova-Lepe}, R.D. {Valle}, and G.~{Robledo}.
\newblock Stability analysis of a self-cycling fermentation model with
  state-dependent impulse times.
\newblock {\em Mathematical methods in the applied sciences}, 37:1460--1475,
  2014.

\bibitem{Fan2007}
G.~Fan and G.S.K. Wolkowicz.
\newblock Analysis of a model of nutrient driven self-cycling fermentation
  allowing unimodal response functions.
\newblock {\em Discrete and continuous dynamical systems}, 8(4):801--831, 2007.

\bibitem{Hsu2019}
T.-H. Hsu, T.~Meadows, L.~Wang, and G.S.K. Wolkowicz.
\newblock Growth on two limiting essential nutrients in a self-cycling
  fermentor.
\newblock {\em Math. BioSci. Eng.}, 16:78--100, 2019.

\bibitem{Hughes1996}
S.M. Hughes and D.G. Cooper.
\newblock Biodegradation of phenol using the self-cycling fermentation process.
\newblock {\em Biotechnology and Bioengineering}, 51:112--119, 1996.

\bibitem{laureni2016}
M.~Laureni, P.~Fal{\aa}s, O.~Robin, A.~Wick, D.G. Weissbrodt, J.L. Nielsen,
  T.A. Ternes, E.~Morgenroth, and A.~Joss.
\newblock Mainstream partial nitritation and anammox: long-term process
  stability and effluent quality at low temperatures.
\newblock {\em Water Research}, 101:628--639, 2016.

\bibitem{Samoilenko1995}
A.~Samoilenko and N.A. Perestyuk.
\newblock {\em Impulsive differential equations}.
\newblock World Scientific, Singapore, 1995.

\bibitem{Sarkis1994}
B.E. Sarkas and D.G. Cooper.
\newblock Biodegradation of aromatic compounds in a self-cycling fermenter.
\newblock {\em The Canadian Journal of Chemical Engineering}, 72(5):874--880,
  1994.

\bibitem{Sauvageau2010a}
D.~Sauvageau and D.~G. Cooper.
\newblock Two-stage, self-cycling process for the production of bacteriophages.
\newblock {\em Microbial Cell Factories}, 9(81), 2010.

\bibitem{Sauvageau2010}
D.~Sauvageau, Z.~Storms, and D.~G. Cooper.
\newblock Sychronized populations of \textit{Escherichia coli} using simplified
  self-cycling fermentation.
\newblock {\em Journal of Biotechnology}, 149:67--73, 2010.

\bibitem{Smith2001}
R.~J. Smith and G.S.K. Wolkowicz.
\newblock Analysis of a model of the nutrient driven self-cycling fermentation
  process.
\newblock {\em Dynamics of continuous, discrete and impulsive systems},
  11:239--265, 2004.

\bibitem{Storms2012}
Z.J. Storms, T.~Brown, D.~Sauvageau, and D.G. Cooper.
\newblock Self-cycling operation increases productivity of recombinent protein
  in \textit{Escherichia Coli}.
\newblock {\em Biotechnology and Bioengineering}, 109(9):2262--2270, 2012.

\bibitem{Strous1998}
M.~Strous, J.J. Heijnen, J.G. Kuenen, and M.S.M. Jetten.
\newblock The sequencing batch reactor as a powerful tool for the study of
  slowly growing anaerobic ammonium-oxidizing microorganisms.
\newblock {\em Appl. Microbiol. Biotechnol.}, 50:589--596, 1998.

\bibitem{Sun2011}
K.~Sun, Y.~Tian, L.~Chen, and A.~Kasperski.
\newblock Universal modelling and qualitative analysis of an impulsive
  bioprocess.
\newblock {\em Computers \& Chemical Engineering}, 35(3):492 -- 501, 2011.

\bibitem{Tilman1982}
D.~Tilman.
\newblock {\em Resource competition and community structure}.
\newblock Princeton University Press, New Jersey, 1982.

\bibitem{Liebig1840}
J.~Von~Liebig.
\newblock {\em Die organische Chemie in ihrer Anwendung auf Agrikultur und
  Physiologie}.
\newblock Friedrich Vieweg, Braunschweig, 1840.

\bibitem{Wang2017}
M.~Wang, J.and~Chae, D.~Sauvageau, and D.~C. Bressler.
\newblock Improving ethanol productivity through self-cycling fermentation of
  yeast: a proof of concept.
\newblock {\em Biotechnology for biofuels}, 10(1):193, 2017.

\bibitem{Wincure1995}
B.M. Wincure, D.G. Cooper, and A.~Rey.
\newblock Mathematical model of self-cycling fermentation.
\newblock {\em Biotechnol. Bioeng.}, 46(2):180--183, 1995.

\end{thebibliography}
\end{document}